\documentclass[11pt]{amsart}

\usepackage[all]{xy}
\usepackage{amsmath,amssymb,amsthm,latexsym,mathdesign}
\usepackage{amscd}
\usepackage{graphics}
\usepackage[dvips]{graphicx}
\usepackage{euscript}
\usepackage{color}
\usepackage[T1]{fontenc}
\usepackage{parskip}
\usepackage[colorlinks]{hyperref}
\usepackage[centertags]{amsmath}
\usepackage{enumitem}
\usepackage{mathtools}
\usepackage{comment}
\usepackage{mathabx}
\usepackage{tikz}
\usepackage{tikz-cd}
\usepackage{textcomp}
\usepackage{wasysym}
\usepackage{wrapfig}

\makeatletter
\newsavebox\myboxA
\newsavebox\myboxB
\newlength\mylenA

\newcommand*\xoverline[2][0.5]{%
    \sbox{\myboxA}{$\m@th#2$}%
    \setbox\myboxB\null
    \ht\myboxB=\ht\myboxA%
    \dp\myboxB=\dp\myboxA%
    \wd\myboxB=#1\wd\myboxA
    \sbox\myboxB{$\m@th\overline{\copy\myboxB}$}
    \setlength\mylenA{\the\wd\myboxA}
    \addtolength\mylenA{-\the\wd\myboxB}%
    \ifdim\wd\myboxB<\wd\myboxA%
       \rlap{\hskip 0.5\mylenA\usebox\myboxB}{\usebox\myboxA}%
    \else
        \hskip -0.5\mylenA\rlap{\usebox\myboxA}{\hskip 0.5\mylenA\usebox\myboxB}%
    \fi}
\makeatother

\usepackage[colorinlistoftodos,prependcaption]{todonotes}

\makeatletter
\def\bign#1{\mathclose{\hbox{$\left#1\vbox to8.5\p@{}\right.\n@space$}}\mathopen{}}
\def\Bign#1{\mathclose{\hbox{$\left#1\vbox to11.5\p@{}\right.\n@space$}}\mathopen{}}
\def\biggn#1{\mathclose{\hbox{$\left#1\vbox to14.5\p@{}\right.\n@space$}}\mathopen{}}
\def\Biggn#1{\mathclose{\hbox{$\left#1\vbox to17.5\p@{}\right.\n@space$}}\mathopen{}}
\makeatother

\newtheorem{theorem}[subsection]{Theorem}
\newtheorem{theorem*}{Theorem}
\newtheorem{corollary}[subsection]{Corollary}
\newtheorem{lemma}[subsection]{Lemma}
\newtheorem{proposition}[subsection]{Proposition}

\theoremstyle{definition}
\newtheorem{definition}[subsection]{Definition}

\newtheorem{example}[subsection]{Example}

\newtheorem{remark}[subsection]{Remark}
\newtheorem*{remark*}{Remark}

\numberwithin{figure}{section}
\numberwithin{table}{section}
\numberwithin{equation}{section}

\newcommand{\Bo}[1]{{\mathbf #1}}
\newcommand{\MCG}{\operatorname{MCG}}
\newcommand{\Ab}{\operatorname{Ab}}
\newcommand{\Aut}{\operatorname{Aut}}
\newcommand{\Wh}{\operatorname{Wh}}
\newcommand{\im}{\operatorname{im}}
\newcommand{\Push}{\operatorname{Push}}
\newcommand{\Sym}{\operatorname{Sym}}

\newcommand{\OP}{\operatorname}

\begin{document}

\title[Aut-norms and Aut-quasimorphisms on $\Bo F_n$ and $\Bo\Gamma_g$]{Aut-invariant norms and Aut-invariant quasimorphisms on free and surface groups}

\author{Michael Brandenbursky}
\author{Micha\l\ Marcinkowski}
\address{Ben Gurion University of the Negev, Israel}
\email{brandens@math.bgu.ac.il}
\address{Ben Gurion University \& Regensburg Universit{\"a}t \& Uniwersytet Wroc\l awski}
\email{marcinkow@math.uni.wroc.pl}
\keywords{free groups, mapping class groups, quasi-morphisms, invariant norms}
\subjclass[2010]{57}

\begin{abstract}
Let $\Bo F_n$ be the free group on $n$ generators and $\Bo\Gamma_g$ the surface group of genus $g$. 
We consider two particular generating sets: the set of all primitive elements in $\Bo F_n$ and 
the set of all simple loops in $\Bo\Gamma_g$.  
We give a complete characterization of distorted and undistorted elements in the corresponding
$\Aut$-invariant word metrics. In particular, we reprove Stallings theorem and 
answer a question of Danny Calegari about the growth of simple loops.
In addition, we construct infinitely many quasimorphisms on $\Bo F_2$ that are $\Aut(\Bo F_2)$-invariant. 
This answers an open problem posed by Mikl\'os Ab\'ert. 
\end{abstract}

\maketitle

\section{Introduction.}
Let $G$ be a group and let $\Aut(G)$ be the group of all automorphisms of $G$. 
A function  $|\mathord{\cdot}|\colon G\to [0,\infty)$ is called a norm
if it satisfies the following conditions for all $g,h\in G$:
\begin{itemize}
\item
$|g|=0$ if and only if $g=1_G$.
\item
$|g|=|g^{-1}|$
\item
$|gh|\leq |g|+|h|$
\end{itemize}

In this paper we study norms that are $\Aut(G)$-invariant, i.e., for each $x\in G$ and $\psi\in\Aut(G)$ we have $|x| = |\psi(x)|$. 
An example of such a norm is the word norm $|\mathord{\cdot}|_S$ defined by $\Aut(G)$-invariant generating set $S$, that is:
$$
|x|_S =\min\{n \colon x = s_1\ldots,s_n,~\text{where}~s_i \in S~\text{for each}~i\}.
$$
Under mild assumptions, up to bi-Lipschitz equivalence, there is only one $\Aut(G)$-invariant word norm. 
We denote it by $|\mathord{\cdot}|_{\Aut}$ and call it the $\Aut$-norm of $G$. 
We focus on two types of groups: 
surface groups, where the $\Aut$-norm counts the minimal number of simple loops needed to express an element in a group, 
and free groups, where the $\Aut$-norm counts the minimal number of primitive elements needed to express an element in a group.

Let $x\in G$. Recall that $x$ is \textbf{undistorted} with respect to $|\mathord{\cdot}|$
if there exists $C>0$ such that $|x^n|>Cn$. Otherwise it is \textbf{distorted}. 
Note that this notion depends only on the bi-Lipschitz class of $|\mathord{\cdot}|$.
In this paper we characterize distorted and undistorted elements in surface and free groups.
The main idea is to find an appropriate quasimorphism on $G$. More precisely,   
in order to show that $x\in G$ is undistorted in $|\mathord{\cdot}|_{\Aut}$, it is enough to find a homogeneous 
quasimorphism which is non-zero on $x$ but is bounded on some $\Aut(G)$-invariant generating set $S$.
This strategy was previously used in the context of conjugation invariant norms, see e.g. \cite{bgkm, BK-disc}.

We say that $G$ satisfies \textbf{bq-dichotomy} with respect to the $\Aut$-norm, 
if for every element $x \in G$, either the cyclic group $\langle x \rangle$ is bounded in the $\Aut$-norm, 
or we can find a homogeneous quasimorphism which does not vanish on $x$ but is bounded
on some $\Aut(G)$-invariant generating set. 
That is, all undistorted elements are detected by appropriate quasimorphisms. 

The main result of this paper is presented below, i.e.,  
we prove the following theorem (see Theorems \ref{thm:free.main} and \ref{thm:surface.main} in the text),
which, in particular, answers (see Corollary~\ref{cor:calegari-question}) 
a question of Danny Calegari from 2007 (see \cite[Question 1.6]{C-08}) and 
gives a simple proof of Stallings theorem \cite[Theorem 2.4]{stallings} on Whitehead graphs of separable elements.

\begin{theorem*}\label{thm:T1}
Surface groups and free groups satisfy bq-dichotomy with respect to their $\Aut$-norms. 
Moreover, in both cases, there is an explicit characterization of undistorted elements. 
\end{theorem*}

This theorem has an application to geodesics on closed hyperbolic surfaces.
More precisely, we show that for every neighborhood $U$ of a point $p$ 
that lies on a closed simple non-separating geodesic $l$, 
there is another simple closed geodesic $l'$ which passes through $U$,
see Theorem \ref{fac:many.fellows.prop.}.

In addition, the methods we use allow us to prove the following theorem (see Corollary~\ref{cor:F_2.inv} in the text) which
answers a question of Mikl\'os Ab\'ert from 2010 (see \cite[Question 47]{abert}) in the case of $\Bo F_2$. 

\begin{theorem*}
The space of homogeneous $\Aut(\Bo F_2)$-invariant quasimorphisms on $\Bo F_2$ is infinite dimensional.
\end{theorem*}

As a corollary we provide an infinite dimensional space of quasimorphisms on $\Bo F_2$ where each
quasimorphism can not be expressed as a finite sum of counting quasimorphisms, see Remark \ref{R:inf-not-Brooks}.

\subsection*{Acknowledgments.}
We would like to thank Mladen Bestvina for answering our questions, 
and Danny Calegari and Hugo Parlier for comments on the early draft of our paper.  

Both authors were partially supported by GIF-Young Grant number I-2419-304.6/2016
and by SFB 1085 ``Higher Invariants'' funded by DFG.

Part of this work has been done during the authors stay at the University of Regensburg and the second author stay
at the Ben-Gurion University. We wish to express our gratitude to both places for the support and excellent working conditions.

\section{Preliminaries}
\subsection{$\Aut$-invariant norm on free groups.}\label{sub:primitive.norm}
Let $\Bo F_n$ be the free group of rank $n$.
An element $b$ is called \textbf{primitive}, if it is an element of some free basis of $\Bo F_n$. 
Note that if $b$ is primitive and $\psi\in \Aut(\Bo F_n)$, then $\psi(b)$ is primitive and each base element has a form 
$\psi(b)$ for some fixed $b$. 
In other words, the set of all primitive elements in $\Bo F_n$ is a single orbit of the $\Aut(\Bo F_n)$ action. 

Given an element of $\Bo F_n$ where $n\geq 3$, it is difficult to decide if it is primitive or not. 
In his celebrated papers (\cite{wh1,wh2}), J.~H.~C.~Whitehead provided an algorithmic method to solve this problem. However, 
the time complexity of this algorithm seems to be ineffective for large $n$ (see \cite{gen.algo} for an attempt to find fast algorithms). 
It is worth to note that in the case of $\Bo F_2$ the situation is completely different. 
There is a quadratic in time algorithm which checks whether an 
element of $\Bo F_2$ is primitive or not. It can be extracted from \cite{prim.in.f2}. 

We consider the following norm: 
$$
|x|_p=\min\{n~|~x=b_1 \ldots b_n, \text{where}~b_i~\text{is a primitive element in}~\Bo F_n\}, 
$$
which we call the \textbf{primitive norm} of $x$.
A priori it is not clear whether this norm is unbounded. 
Indeed, for the free group of infinite rank, the norm of any element is at most $2$.
In case of arbitrary finite rank, the unboundedness was proven in \cite{prim.width} by exhibiting a relevant non-trivial
homogeneous quasimorphism. Our Theorem~\ref{thm:free.main} may be viewed as a generalization of this result.
In the proof we use Whitehead graphs and Whitehead automorphisms, and 
as a corollary we obtain a result due to Stallings 
(See Corollary~\ref{C:separable}) on Whitehead graphs of separable elements. 

\subsection{$\Aut$-invariant norm on surface groups.}\label{sub:simple.loops.norm}

Let us denote by $S_g$ the oriented closed surface of genus $g$ and let $\star \in S_g$ be a base point.
Let $\Bo\Gamma_g = \pi_1(S_g,\star)$.
An element $s$ of $\Bo\Gamma_g$ is called \textbf{simple}, 
if it can be represented by a based loop with no self-intersection points (such a loop is called simple). 
Simple elements generate $\Bo\Gamma_g$, thus we can consider the following norm: 
$$
|x|_s=\min\{n~|~x=s_1 \ldots s_n, \text{where}~s_i~\text{is a simple element in}~\Bo\Gamma_g\}.
$$
We call $|\mathord{\cdot}|_s$ the \textbf{simple loops norm}. 
Danny Calegari proved that every non-simple element $x \in\Bo\Gamma_g$ is undistorted in this norm \cite{C-08}, 
leaving the case of simple elements open. Theorem~\ref{thm:surface.main} solves this case.
In particular, we show that an element which is represented by simple separating closed curve is undistorted.

Recently, Erlandsson considered generating sets consisting of simple loops and obtained interesting results about 
intersection numbers \cite{Erl-1, Erl-2}. However, generating sets she considers are finite and we consider infinite 
generating sets.

Let us put an arbitrary hyperbolic metric on $S_g$. 
It is known that in every homotopy class of a based loop there is a unique based closed geodesic. 
Using Theorem~\ref{thm:surface.main} we draw some conclusions concerning the behavior of closed geodesic in $S_g$,
see Subsection \ref{ss:applications}.  

\subsection{Generalization: $H$-norms.}
\label{sub:general}
Let $G$ be a group and $H<\Aut(G)$ be a subgroup of the group of automorphisms of $G$. 
The group $H$ acts on $G$ from the left.  
We say that 
\begin{itemize}
\item $G$ is $H$-{\bf generated} by $S\subset G$ if  $\bar{S} = \{h(s)\colon s\in S\cup S^{-1}, h\in H\}$ generates $G$.
\item $H$-invariant subset $\bar{S}$ of $G$ is $H$-{\bf finite} if $\bar{S}/H$ is finite,
i.e., $\bar{S}$ is a sum of finitely many $H$-orbits. 
\item $G$ is $H$-{\bf finitely generated} if there exists an $H$-finite subset of $G$ generating $G$.
Equivalently, if there exists a finite set $S \subset G$ such that 
$\bar{S} = \{h(s) \colon s \in S \cup S^{-1}, h \in H\}$ generates $G$. 
\end{itemize}
Having an $H$-finite set $\bar{S}$ which generates $G$, 
we consider the word norm $|\mathord{\cdot}|_{\bar{S}}$ on $G$ defined by $\bar{S}$.  
Note, that $|\mathord{\cdot}|_{\bar{S}}$ is $H$-invariant, i.e.,
$|h(g)|_{\bar{S}} = |g|_{\bar{S}}$ for $h \in H$ and $g \in G$.

The norm $|\mathord{\cdot}|_{\bar{S}}$ depends on the choice of $H$-finite set $\bar{S}$. 
However, as long as $\bar{S}$ is $H$-finite, $|\mathord{\cdot}|_{\bar{S}}$ belongs to the same bi-Lipschitz equivalence class.
We denote this equivalence class by $|\mathord{\cdot}|_H$.
The norm $|\mathord{\cdot}|_H$ is maximal among all $H$-invariant norms on $G$, namely: 
for every $H$-invariant norm $|\mathord{\cdot}|$, there exists $C>0$ such that $|x| < C|x|_H$ for every $x \in G$. 

Examples include (we always assume that $G$ is $H$-finitely generated):
\begin{enumerate}
\item $H$ is trivial, then $|\mathord{\cdot}|_H$ is the standard word norm.
\item $H$ is the group of inner automorphisms of $G$, then $|\mathord{\cdot}|_H$ is the conjugation invariant word norm, see e.g. \cite{bgkm}.
\item $H$ is the full automorphism group, then $|\mathord{\cdot}|_H$ is called the $\Aut$-norm on $G$. 
\end{enumerate}

Let us show that the primitive norm and the simple loops norm are examples of $\Aut$-norms on free and surface groups.
It is clear that the primitive norm is an $\Aut$-norm since the set of primitive elements in $\Bo F_n$ equals to the set
$$\{\psi(b)~|~\psi \in\Aut(\Bo F_n)\},$$ 
where $b$ is an arbitrary chosen primitive element.
Thus this set is a single $\Aut(\Bo F_n)$-orbit. 
In the case of the simple loops norm we use the Baer-Dehn-Nielsen theorem which states:
$$\Aut(\Bo\Gamma_g)\cong\MCG^{\pm}(S_g,\star).$$
Now it follows from the classification of surfaces, 
that $\Aut(\Bo\Gamma_g)$-orbit of a simple element $s\in \Gamma_g$ is determined by the homeomorphism 
type of the surface $S_g \backslash \gamma_s$, where $\gamma_s$ is the corresponding simple loop. 
Since there are only finitely many homeomorphism types of such surfaces, 
the set of simple elements consists of finitely many $\Aut(\Bo\Gamma_g)$-orbits.

\subsection{Quasimorphisms and distortion.}
Let us recall a notion of a quasimorphism. A function $q\colon G\to\Bo R$ 
is called a \textbf{quasimorphism} if there exists $D$ such that 
$$|q(a)-q(ab)+q(b)|<D$$
for all $a,b\in G$. The minimal such $D$ is called the {\bf defect} of $q$ and denoted by $D_q$.
A quasimorphism $q$ is \textbf{homogeneous} if $q(x^n)=nq(x)$ for all $n \in\Bo Z$ and all $x \in G$. 
Homogeneous quasimorphisms are constant on conjugacy classes, i.e., $q(x) = q(yxy^{-1})$ for all $x,y\in G$.
We refer to \cite{scl} for further details.

\begin{lemma}\label{lem:qm.undist}
Let $G$ be a group and let $H <\Aut(G)$. Assume that $\bar{S}~\subset~G$ is an 
$H$-finite subset which generates $G$.
Let $x \in G$. If there exists a homogeneous quasimorphism $q\colon G \to\Bo R$ such that $q(x)\neq 0$ and $q$ is 
bounded on $\bar{S}$, then $x$ is undistorted in the norm $|\mathord{\cdot}|_H$. 
\end{lemma} 

\begin{proof}
Let $w \in G$ and $C$ be such that $|q(s)|<C$ for all $s \in \bar{S}$. 
We can write $w$ in the following form: $w = s_1\ldots s_n$ where $n = |w|_H$ and $s_i \in \bar{S}$. 
The following inequality holds:
$$
|q(w)|\leq\sum_1^n |q(s_i)|+(n-1)D_q\leq nC+(n-1)D_q \leq (C+D_q)|w|_H.
$$
We apply this to $x^n$ and get  
$$n|q(x)|=|q(x^n)|\leq (C+D_q)|x^n|_H.$$
Since $q(x) \neq 0$, $|x^n|_H$ growths linearly with $n$. 
\end{proof}

\begin{remark}\label{rem:qm.undist}\rm
In Lemma~\ref{lem:qm.undist}, instead of assuming that $q$ is homogeneous, it is enough to assume that $q(x^n)$ growths linearly. 
\end{remark}

\begin{corollary}\label{cor:H-inv}
Assume that $q$ is an $H$-invariant homogeneous quasimorphism and $q(x)\neq 0$. 
Then $x$ is undistorted in $|\mathord{\cdot}|_H$. 
\end{corollary}
 
Let $q\colon G\to\Bo R$ be a quasimorphism. We define 
$$
\bar{q}(a)=\lim_{n\to\infty}q(a^n)/n.
$$
Straight-forward computations show that $\bar{q}$ is a homogeneous quasimorphism, see e.g. \cite{scl}. 
The following lemma relates general quasimorphisms with homogeneous quasimorphisms. 

\begin{lemma}[\cite{scl}]\label{lem:homo}
Let $q\colon G\to\Bo R$ be a quasimorphism such that $q(x^n)$ growths linearly.
Then a homogeneous quasimorphism $\bar{q}\colon G\to \Bo R$, called the homogenization of~$q$,  satisfies
$|q(a)-\bar{q}(a)| \leq D_q$ for all $q \in G$, and $\bar{q}(x) \neq 0$. 
\end{lemma}

\section{$\Aut$-norms on free groups.}
In this section we use counting quasimorphisms in order to investigate distortion in free groups.

\subsection{Whitehead graph.} 
Let $\Bo F$ be a free group (possibly not finitely generated). Let $w \in\Bo F$ and let $B$ be any free basis of $\Bo F$.
Here we use the convention that if $b \in B$, then $b^{-1} \notin B$. 
Suppose that $w$ is a cyclically reduced word with respect to the basis $B$ and $w = w_1w_2\ldots w_n$ is the reduced expression of $w$ in $B$. 
We define the Whitehead graph $\Wh_B(w)$ as follows: for each element $b\in B$ we have two vertices in $\Wh_B(w)$ labeled by $b$ and $b^{-1}$.  
For every two consecutive letters $w_i,w_{i+1}$ in $w$, we draw an edge from $w_i$ to $w_{i+1}^{-1}$. 
We regard $w_n,w_0$ as being consecutive in $w$, that is, we have an edge from $w_n$ to $w_0^{-1}$. 
If $w$ is not reduced with respect to $B$, then denote by $r_B(w)$ the unique cyclically reduced word in the
conjugacy class of $w$. We define $\Wh_B(w):=\Wh_B(r_B(w))$.  

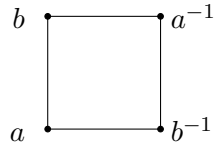
\begin{figure}[h]
\begin{tikzpicture}[scale=1.5]
\draw[fill] (0,0) circle [radius=0.025];
\node [left] at (0,0) {$a^{\phantom{1}}$};
\draw[fill] (0,1) circle [radius=0.025];
\node [left] at (0,1) {$b^{\phantom{1}}$};
\draw[fill] (1,0) circle [radius=0.025];
\node [right] at (1,0) {$b^{-1}$};
\draw[fill] (1,1) circle [radius=0.025];
\node [right] at (1,1) {$a^{-1}$};
\draw (0,1) -- (1,1);
\draw (0,0) -- (1,0);
\draw (0,0) -- (0,1);
\draw (1,0) -- (1,1);
\end{tikzpicture}
\caption{$\Wh_{\{a,b\}}(aba^{-1}b^{-1}).$}
\end{figure}

\begin{theorem}[Whitehead \cite{wh1}]\label{whitehead}
If $b\in\Bo F$ is a primitive element and $\Wh_B(b)$ is connected, then $\Wh_B(b)$ has a cut-vertex.
\end{theorem}

A vertex $v$ of a graph is a \textbf{cut-vertex} if after removing $v$, the graph has more connected components.   
In the example below $a$, $b$, $a^{-1}$ and $b^{-1}$ are cut-vertices. 

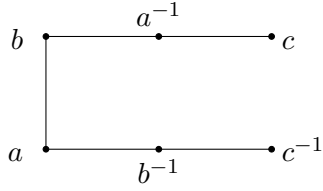
\begin{figure}[h]
\begin{tikzpicture}[scale=1.5]

\draw[fill] (2,0) circle [radius=0.025];
\node [right] at (2,0) {$c^{-1}$};
\draw[fill] (2,1) circle [radius=0.025];
\node [right] at (2,1) {$c^{\phantom{1}}$};
\draw[fill] (0,0) circle [radius=0.025];
\node [left] at (0,0) {$a^{\phantom{1}}$};
\draw[fill] (0,1) circle [radius=0.025];
\node [left] at (0,1) {$b^{\phantom{1}}$};
\draw[fill] (1,0) circle [radius=0.025];
\node [below] at (1,0) {$b^{-1}$};
\draw[fill] (1,1) circle [radius=0.025];
\node [above] at (1,1) {$a^{-1}$};
\draw (0,1) -- (1,1);
\draw (0,0) -- (1,0);
\draw (0,0) -- (0,1);
\draw (1,1) -- (2,1);
\draw (1,0) -- (2,0);
\end{tikzpicture}
\caption{$\Wh_{\{a,b,c\}}(aba^{-1}b^{-1}c).$}
\end{figure}

\begin{lemma}\label{L:p^2}
Let $B$ be a free basis of $\Bo F$ and let $b \in\Bo F$ be a cyclically reduced base element. 
Assume that $x\in\Bo F$ is a cyclically reduced word such that its Whitehead graph $\Wh_B(x)$ is connected and has no cut-vertices. 
Then the reduced expression of $b$ in the basis $B$ does not contain the reduced expressions of $x^2$ and of $x^{-2}$ as subwords.
\end{lemma}

\begin{proof}
Let us fist notice, that if a reduced expression of $b$ contains the reduced expression of $x^2$, then 
$\Wh_B(b)$ contains $\Wh_B(x)$ as a subgraph. 
Indeed, every edge of $\Wh_B(x)$ is given by two cyclically consecutive letters in $x$.
If $x^2$ is a subword of $b$, we can find all those edges in $\Wh_B(b)$.
$\Wh_B(x)$ has no cut-vertices, it is connected and has the same vertex set as $\Wh_B(b)$.
Thus if $\Wh_B(x)$ is a subgraph of $\Wh_B(b)$, then $\Wh_B(b)$ is a connected graph with no cut-vertices. 
This contradicts the Whitehead Theorem~\ref{whitehead}.
\end{proof}

\begin{remark}\label{R:p^2}
In Lemma~\ref{L:p^2}, the assumption that $b$ is cyclically reduced can not be omitted. 
For general $b$ one can show that the number of occurrences of $x^2$ minus 
the number of occurrences of $x^{-2}$ is bounded by $2|x|_B$.
Here $|\mathord{\cdot}|_B$ is the word metric defined by $B$.
We would like to add that a variation of Lemma~\ref{L:p^2} already appeared in~\cite{prim.width}.
\end{remark}

\subsection{Counting quasimorphisms.}

Let $w\in\Bo F$ and let $B$ be a basis of $\Bo F$. 
Having an element $x\in\Bo F$, we can write $x$ and $w$ as reduced expressions in the base $B$.
We define $C_w(x)$ to be the number of occurrences of $w$ as a subword of $x$ (the subwords can overlap). 
In \cite{brooks} Brooks proved that the following function
$$
Br_w(x) = C_w(x) - C_w(x^{-1})
$$
is a quasimorphism. See also \cite{Rhem}. Usually we suppress the basis $B$ from the notation. 

\begin{lemma}\label{L:brooks}
Let $x\in\Bo F$ be a cyclically reduced word.
Assume that $\Wh_B(x)$ is connected with no cut-vertices.
Then $x$ is undistorted in the $\Aut$-norm.
Moreover, there exists a homogeneous quasimorphism which is bounded on primitive elements and is non-zero on $x$.
\end{lemma}

\begin{proof}
Let $Br_{x^2}$ be a counting quasimorphism defined with respect to the basis $B$. 
Let $b\in\Bo F$ be a primitive element. By Lemma~\ref{L:p^2} and Remark \ref{R:p^2} we have that 
$$|Br_{x^2}(b)| \leq 2|x|_B.$$
Moreover, $|Br_{x^2}(x^n)|$ growths linearly with $n$, thus by Remark \ref{rem:qm.undist}, $x$ is undistorted in the $\Aut$-norm. 
In order to obtain a desired quasimorphism, it is enough to consider the homogenization of $Br_{x^2}$, see Lemma \ref{lem:homo}.
\end{proof}

\subsection{Separable sets.}
Let $A$ be a finite subset of $\Bo F$. We call $A$ \textbf{separable}, 
if there exist two non-trivial free factors $F_1$,$F_2$ such that $\Bo F=F_1\ast F_2$,
and every element in $A$ can be conjugated into $F_1$ or $F_2$. 
We always assume that elements in $A$ are cyclically reduced.
Note that $A$ is separable if and only if $\phi(A)$ is separable, where $\phi\in\Aut(\Bo F)$. 
The following proposition follows from Proposition~2.2 and Proposition~2.3 in \cite{stallings}.

\begin{proposition}\label{P:stal}
Let $x \in\Bo F$. If $\Wh_B(x)$ is disconnected, then $x$ is separable. 
If $\Wh_B(x)$ is connected and has a cut-vertex, then there is $\psi\in \Aut(\Bo F)$, such that
$|\psi(x)|_B < |x|_B$.
\end{proposition}

\subsection{Distortion of elements in the $\Aut$-norm of free groups.}
\begin{theorem}\label{thm:free.main}
Let $\Bo F$ be a free group and let $x\in\Bo F$.
Then either 
\begin{enumerate}[leftmargin=0.5cm,itemindent=.5cm,labelwidth=\itemindent,labelsep=0cm,align=left,label=\alph*)]
\item $x$ is separable and then the cyclic subgroup generated by $x$ is bounded in the $\Aut$-norm, or
\item $x$ is undistorted in the $\Aut$-norm.  Moreover, there exists a homogeneous quasimorphism 
which is bounded on the set of all primitive elements and is non-trivial on $x$.
\end{enumerate}
\end{theorem}

\begin{proof}
Let $|\mathord{\cdot}|_p$ be the primitive norm, see Subsection \ref{sub:primitive.norm}. 
Suppose that $x$ is separable. It means that $\Bo F=F_1\ast F_2$, $F_2$ is not trivial and $x$ can be conjugated into $F_1$. 
Since $|\mathord{\cdot}|_p$ is invariant under inner automorphisms, we can assume that $x\in F_1$. 
Let $B_1$ and $B_2$ be some bases of $F_1$ and $F_2$ respectively. Let $p\in B_2$ and $n\in\Bo N$. 
Note that the element $x^np$ is primitive. Indeed, the set $B_1\cup\{x^np\}\cup (B_2\backslash \{p\})$ is a free basis of $\Bo F$. 
Thus $|x^n|_p = |x^npp^{-1}|_p\leq 2$.

Now suppose, that $x$ is not separable. Let $B$ be a basis of $\Bo F$.
We claim that we can find an element $y$ in the  $\Aut(\Bo F)$-orbit of $x$ such that 
$\Wh_B(y)$ is connected and has no cut-vertices.
Note that $\Wh_B(x)$ is connected. Indeed, if it was not connected, 
then by Proposition~\ref{P:stal}, $x$ would be separable. 
Let $|\mathord{\cdot}|_B$ be the word norm defined by $B$.
If $\Wh_B(x)$ has a cut-vertex, then again by Proposition~\ref{P:stal} 
there is an automorphism $\psi \in\Aut(\Bo F)$ such that $|\psi(x)|_B < |x|_B$. 

Now we consider $\Wh_B(\psi(x))$. As before, it is a connected graph. 
If it has a cut-vertex, then we apply Proposition~\ref{P:stal} again and find a new $\psi$
and further reduce the length of the element.  
At the end we get an element $y$ 
with the property that $\Wh_B(y)$ is connected with no cut-vertices. 
Now consider the cyclical reduction $y'=r_B(y)$. By definition we have $\Wh_B(y') = \Wh_B(y)$. 
Lemma~\ref{L:brooks} gives us a homogeneous quasimorphism~$q$ which is
bounded on primitives and is non-trivial on~$y'$.
Since $x$ and $y'$ are in the same $\Aut$-orbit, there exists $\psi\in\Aut(\Bo F)$ such that $\psi(x) = y'$.  
If we define $q'(w) = q(\psi(w))$, then $q'$ is bounded on primitive elements and $q'(x)\neq 0$. 
\end{proof}

\begin{corollary}\label{C:separable}
An element $x$ of $\Bo F$ is separable if and only if  
$\Wh_B(x)$ is not connected or has a cut-vertex for every basis $B$.  
\end{corollary}

\begin{proof}
Assume that $x$ is separable. If there is a basis $B$ such that $\Wh_B(x)$ is connected and has no cut-vertices,
then by applying Lemma~\ref{L:brooks} for an element $r_B(x)$, 
we see that $x$ is undistorted in the $\Aut$-norm. Thus $x$ is not separable by Theorem~\ref{thm:free.main}.

To prove the reversed implication, we apply inductively Proposition~\ref{P:stal} 
and obtain an element $y$ in the same $\Aut(\Bo F)$-orbit as $x$, such that $\Wh_B(y)$ is not connected. 
Thus again by Proposition~\ref{P:stal}, an element $y$, and consequently $x$, is separable. 
\end{proof}

\begin{remark}
Corollary~\ref{C:separable} is not entirely new. It can be deduced from Theorem~2.4 and Proposition~2.3 in \cite{stallings}.
However, we think that our proof of this fact is interesting since it is simpler and shorter than the proof of Stallings.
\end{remark}

\section{$\Aut$-norm on surface groups.}

In this section we study distortion in surface groups. Our main tool is the theory of mapping class groups.
The principal idea is to embed a surface group in its automorphism group (which is a mapping class group of a surface) 
via the Birman embedding, and then find appropriate quasimorphisms on this group.
In Subsection \ref{sub:nielsen.thurston} we recall the Nielsen-Thurston normal form of a mapping class.
Then in Subsection \ref{sub:gen.kra} we give the Nielsen-Thurston decomposition of mapping classes which lie in the image of the Birman embedding. 
Finally, in Subsection \ref{sub:surface.main} we use the quasimorphisms defined by Bestvina-Bromberg-Fujiwara to prove Theorem \ref{thm:T1} for surface groups. 

\subsection{Nielsen-Thurston normal form.}\label{sub:nielsen.thurston} 
Let $S$ be a compact oriented surface with finitely many punctures in the interior of $S$.
By $\MCG(S)$ we denote the mapping class group of $S$, that is the group of isotopy classes of orientation preserving
homeomorphisms of $S$. We assume that homeomorphisms and isotopies fix the boundary of $S$ pointwise.

We recall briefly the Nielsen-Thurston normal form of an element in $\MCG(S)$. 
A loop $\gamma$ in $S$ is called \textbf{essential}, if no component of $S \backslash \gamma$ is homeomorphic to a disc, 
a punctured disc or an annulus.   
An element $g\in\MCG(S)$ is called \textbf{reducible} if there exists a non-empty set $C = \{c_1,\ldots,c_n\}$ of isotopy classes
of essential simple loops in $S$ such that: 
\begin{enumerate}
\item All elements in $C$ can be represented by pairwise disjoint simple loops.
\item The set $C$ is $g$-invariant. 
\end{enumerate}

Such $C$ is called a \textbf{reduction system} for $g$. 
A reduction system for $g$ is maximal if it is not a proper subset of any other reduction system for $g$. 
There may be many maximal reduction systems. However, we can define the unique one by defining the 
\textbf{canonical reduction system} to be the intersection of all maximal reduction systems. 
Note that the canonical reduction system is not necessary maximal.

Now let us describe the canonical form of an element of a mapping class group. 
Assume for a moment, that $S$ has no boundary.
Let $g\in\MCG(S)$ and let $C=\{c_1,\ldots,c_m\}$ be its canonical reduction system.
Choose pairwise disjoint representatives of the classes $c_i$ together with pairwise disjoint closed annuli $R_1,\ldots,R_m$, 
where $R_i$ is a closed neighborhood of a representative of $c_i$.
Let $S_1,\ldots,S_p$ be the closures of connected components of $S \backslash \bigcup_{i=1}^m R_i$.
Then there is a power $k$ and a representative $\psi\in\operatorname{Homeo}^+(S)$ of $g^k$ such that: 
\begin{enumerate}
\item The homeomorphism $\psi$ fixes the subsurfaces, i.e., $\psi(R_i)=R_i$ for $ 1\leq i\leq m$ and $\psi(S_i)=S_i$ for $1\leq i\leq p$.
\item The restriction of $\psi$ to $R_i$ is a power of a Dehn twist.
\item The restriction of $\psi$ to $S_i$ is pseudo-Anosov or the identity.
\end{enumerate}

Thus, up to finite power, any element is described as a commuting product of 
powers of Dehn twists and pseudo-Anosov homeomorphisms on subsurfaces.
If $S$ has a boundary, then we need to add to our collection $R_i$  
collar neighborhoods of boundary curves. 
Then the additional terms which can appear in the decomposition of $g^k$ are powers of Dehn twists along boundary curves. 
A mapping class which has only one factor in the Nielsen-Thurston decomposition is called \textbf{pure}. 
We have the following characterization of the canonical reduction system.  

\begin{proposition}\label{prop:crs.char}
Assume that the surface $S$ has no boundary and possibly has punctures.  
The system $C=\{c_1,\ldots,c_s\}$ is the canonical reduction system for $g$ if and only if the following two conditions hold: 
\begin{enumerate}[itemsep=-0.2cm,leftmargin=0.5cm,itemindent=.5cm,labelwidth=\itemindent,labelsep=0cm,align=left,label=\alph*)]
\item There exists $k \in\Bo N$ and pairwise disjoint loops $\gamma_i$ representing classes $c_i$ such that 
$g^k$ restricted to any component of $S \backslash (\gamma_1 \cup \ldots \cup \gamma_s)$ 
is trivial or pseudo-Anosov. 
\item The set $C$ is a minimal set with this property.
\end{enumerate}
\end{proposition}

Dehn twists are not mentioned in Proposition~\ref{prop:crs.char} since every Dehn twist along some $\gamma_i$ 
becomes trivial in the mapping class group of $S \backslash (\gamma_1 \cup \ldots \cup \gamma_s)$. 

\subsection{Filling curves and the theorem of Kra.}\label{sub:gen.kra}
Let $S$ be a closed oriented surface of genus $g$, possibly with punctures, and let $\star\in S$.
Assume that $S$ has negative Euler characteristic. We consider the the Birman exact sequence:
$$
1 \to \pi_1(S,\star) \xrightarrow{\Push} \MCG(S,\star) \xrightarrow{F} \MCG(S) \to 1.
$$
By $\MCG(S,\star)$ we mean the group of homeomorphisms which fix~$\star$, 
taken up to isotopies which fix~$\star$ at any time. 
Since fixing a point and removing a point does not make any difference for mapping classes, 
we have that $\MCG(S,\star)=\MCG(S\backslash\star)$.
The map $F$ is the forgetful map. The map $\Push$ is defined as follows:
let $\gamma$ be a based loop which represents an element in $\pi_1(S,\star)$. 
Let $\psi$ be any homeomorphism which fixes $\star$, such that: 
there exists an isotopy $H \colon [0,1] \times S \to S$ such that
$H(0,\mathord{\cdot})=\OP{id}_S, H(1,\mathord{\cdot})=\psi$ and $H(t,\star) = \gamma(t)$. 
Then $\Push([\gamma])=[\psi]$.
This is a well-defined map. One can imagine that $\Push([\gamma])$ takes $\star$ and pushes it along the loop $\gamma$. 
For a detailed discussion see \cite{fm}. 

The goal of this subsection is to understand the Nielsen-Thurston 
decomposition of $\Push([\gamma])$. In Theorem \ref{thm:gen.kra} we generalize a theorem of I.~Kra \cite{kra}, for the short proof see \cite{fm}. 
It states that if $\gamma$ is filling (see Definition \ref{def:filling}), then $\Push([\gamma])$ is pseudo-Anosov. 

\begin{example}\label{e:push.simple}
Let $\gamma$ be a simple loop based at $\star$. 
We identify a tubular neighborhood of $\gamma$ with the annulus $\Bo S^1 \times [-1,1]$.
Let $\gamma^{\pm} = S^1 \times \{\pm1\}$. Then $\Push([\gamma]) = T_{\gamma^+}T_{\gamma^-}^{-1}$,
where $T_{\gamma^{\pm}}$ is the Dehn twist along $\gamma^{\pm}$. 
The surface $S$ is not a torus. Hence $\gamma^-$ is not homotopic to $\gamma^+$ in 
$S\backslash\{\star\}$, and $T_{\gamma^+}$ and $T_{\gamma^-}$ are
different elements of $\MCG(S,\star)$. 
\end{example}

In what follows, a subsurface $S'$ of $S$ always assumed to be closed in $S$, 
and the boundary of $S'$ is a union of pairwise disjoint simple loops.
A loop is called \textbf{primitive} if, as an element of the fundamental group it is not a proper power of any other element. 
We always assume, that a loop in $S$ is in general position, i.e., it is a smooth immersed loop with
only double self-intersections. A loop is in a \textbf{minimal position} if it has the minimal number of double points. 

\begin{definition}
Assume that $S$ is a closed oriented surface, possibly with punctures. 
Let $\gamma \colon \Bo S^1 \to S$ be a loop in a minimal position.
We define a \textbf{subsurface $S_{\gamma}$} in the following way:
First we consider a small collar neighborhood 
$$N \colon \Bo S^1 \times [-1,1] \to S,$$ 
where $N$ is a smooth immersion and $N(\cdot,0) = \gamma(\cdot)$, 
such that the image of $N$ retracts onto the image of $\gamma$.
Then we add to the image of $N$ all the components 
of $S \backslash \im(N)$ which are disks or punctured disks.
\end{definition}

\begin{lemma}\label{lem:essential}
Assume that $S$ is a closed oriented surface, possibly with punctures. 
Let $\gamma$ be a loop in a minimal position such that it is not homotopic to a power of a simple loop. 
Then the boundary components of $S_{\gamma}$ are essential simple loops in $S$. 
\end{lemma}
\begin{proof}
Assume that one boundary component $\partial_0 S_{\gamma}$ is not essential in $S$.
Then $\partial_0 S_{\gamma}$ bounds a disc or a punctured disc $D$. 
Thus $S\setminus\partial_0 S_{\gamma}$ has two connected components: $D$ and $S\setminus D$.
Hence $\gamma$ is contained either in $D$ or in $S\setminus D$.
By construction, it is impossible that $\gamma\subset S\setminus D$, because then we would add $D$ to $S_{\gamma}$. 
Thus $\gamma \subset D$. But in $D$ every loop is homotopic to a power of a simple loop, and we get a contradiction. 
\end{proof}

\begin{lemma}\label{lem:annulus}
Assume that $S$ is a closed oriented surface, possibly with punctures. 
Let $\gamma$ be a non-trivial loop in a minimal position such that is not homotopic to a power of a simple loop. 
Let $\partial_1 S_{\gamma}$ and $\partial_2 S_{\gamma}$ be two homotopic components of $\partial S_{\gamma}$.
Then they bound an annulus lying outside of the interior of $S_{\gamma}$.
Moreover, they are not homotopic to any other component of $\partial S_{\gamma}$.
\end{lemma}
\begin{proof}
Any two homotopic simple loops bound an annulus in $S$. 
Let $A$ be an embedded annulus bounded by $\partial_1 S_{\gamma}$ and $\partial_2 S_{\gamma}$. 
Since $A$ is a connected component of $S \backslash (\partial_1 S_{\gamma} \cup \partial_2 S_{\gamma})$, 
the loop $\gamma$ is either outside $A$ or inside $A$. 
The loop $\gamma$ cannot be inside $A$, since every loop in $A$ is a power 
of a simple loop and this is excluded by the assumption.

Now assume that some boundary component $\partial_i S_{\gamma}$ of $S_{\gamma}$ is homotopic to $\partial_1 S_{\gamma}$.
Let $\partial_1^{'} S_{\gamma}$ denote a loop in the interior of $A$ homotopic to $\partial_1 S_{\gamma}$.
Let $A'$ be an embedded annulus bounding $\partial_1^{'} S_{\gamma}$ and $\partial_i S_{\gamma}$.
If $\partial_i S_{\gamma}$ is not $\partial_1 S_{\gamma}$ or $\partial_2 S_{\gamma}$, then
$\partial_1 S_{\gamma}$ or $\partial_2 S_{\gamma}$ are contained in the interior of $A'$.
This is impossible, since we already know that the interior of $A'$ is disjoint from $S_{\gamma}$.
Thus $\partial_i S_{\gamma}$ equals to $\partial_1 S_{\gamma}$ or $\partial_2 S_{\gamma}$.
\end{proof}

In the following definitions and Lemma~\ref{lem:hass.scott}, 
$S$ is an orientable surface, possibly with boundary and punctures.

\begin{definition}\label{def:filling}
Let $\gamma$ be a loop on $S$. 
We say that $\gamma$ \textbf{fills} $S$, if every essential simple loop on $S$ 
has a non-trivial intersection with every curve homotopic to~$\gamma$.  
\end{definition}

\begin{definition} 
A loop $\gamma\colon\Bo S^1 \to S$ has an embedded $1$-gon if there is a closed arc $\alpha \subset \Bo S^1$, such that 
$\gamma$ restricted to the interior of $\alpha$ is an embedding and $\gamma(\partial\alpha)$ is a single point.
A loop $\gamma\colon \Bo S^1\to S$ has an embedded $2$-gon if there are two disjoint closed arcs $\alpha$, $\beta$ such that 
$\gamma(\partial\alpha)=\gamma(\partial\beta)$ and $\gamma$ restricted to $\alpha$ and $\beta$ is an embedding. 
\end{definition}

We need the following lemma.

\begin{lemma}[Lemma 2.8, \cite{hass.scott}]\label{lem:hass.scott}
Assume that $\gamma$ has no embedded $1$-gons and $2$-gons. 
Let $c$ be an essential simple loop and assume that $c$ is disjoint from some curve homotopic to $\gamma$. 
Then there is a simple loop $c'$ isotopic to $c$ such that $\gamma$ is disjoint from $c'$.   
\end{lemma}

The next two lemmas deal with the surface $S_{\gamma}$. 
Since we defined $S_{\gamma}$ only for surfaces with no boundary,
we assume below that $\partial S = \emptyset$.

\begin{lemma}\label{L:filling}
Let $\gamma$ be a loop in a minimal position on a closed orientable surface $S$, possibly with punctures. 
Then $\gamma$ fills~$S_{\gamma}$.
\end{lemma}
\begin{proof}
By the definition of $S_{\gamma}$ we see that every connected component of $S_{\gamma} \backslash \gamma$
is a disc or a punctured disc. It follows that every essential simple loop intersects $\gamma$ non-trivially.
If there is an essential simple loop $c$ disjoint from some curve homotopic to $\gamma$, then by Lemma~\ref{lem:hass.scott}
we get some essential simple loop $c'$ disjoint from $\gamma$, which is a contradiction. 
\end{proof}

\begin{lemma}\label{L:Euler-neg}
Assume that $\gamma$ is in a minimal position and is not homotopic to a power of a simple loop. 
Then the Euler characteristic of $S_{\gamma}$ is negative.
\end{lemma}
\begin{proof}
On annulus, torus, disk, punctured disk or 2-punctured sphere, every loop is homotopic to a power of a simple loop. 
Thus $S_{\gamma}$ is none of those. 
Every other oriented surface has negative Euler characteristic. 
\end{proof}

Now we prove the main result of this subsection. 

\begin{theorem}[Generalized Kra's Theorem]\label{thm:gen.kra}
Let $S$ be a closed oriented surface (possibly with punctures) whose Euler characteristic is negative.
Let~$\star$ be a base point. Then:
\begin{enumerate}[leftmargin=0.5cm,itemindent=.5cm,labelwidth=\itemindent,labelsep=0cm,align=left,label=\alph*)]
\item If $x\in \pi_1(S,\star)$ is a power of a simple element represented by $\gamma^n$, where $\gamma$ is a simple essential loop in $S$,
then $\Push(x) = T_{\gamma^+}^nT_{\gamma^-}^{-n}$ is the Nielsen-Thurston decomposition of $\Push(x) \in\MCG(S,\star)$.
\item If $x\in \pi_1(S,\star)$ is not a power of a simple element,
then the Thurston-Nielsen decomposition of $\Push(x) \in\MCG(S,\star)$ consists of a single pseudo-Anosov component.
\end{enumerate}
\end{theorem}

\begin{proof}
If $x$ is a power of a simple element, then by Example \ref{e:push.simple}, 
$$\Push(x) = T_{\gamma^+}^nT_{\gamma^-}^{-n}.$$
Since $S$ is not a torus and $\gamma$ is by assumption essential, 
$\gamma^+$ and $\gamma^-$ are not homotopic essential disjoint simple loops. 
If an element of a mapping class group can be represented as a product of powers 
of Dehn twists along disjoint not homotopic essential simple loops,
then this representation is unique.
It follows that the set $C~=~\{\gamma^+,\gamma^-\}$ satisfies the conditions of Proposition~\ref{prop:crs.char}.
Indeed, if $C$ is smaller, then $\Push(x)$ equals to a power of $T_{\gamma^{\pm}}$, and 
hence is not equal to $T_{\gamma^+}^nT_{\gamma^-}^{-n}$.

Assume that $x$ is not a power of a simple element. 
Let $\gamma$ be a loop in a minimal position which represents $x$.
Then by Lemma~\ref{L:Euler-neg} the Euler characteristic of $S_{\gamma}$ is negative and
by Lemma~\ref{L:filling} $\gamma$ fills $S_{\gamma}$. 
By the result of Kra (\cite{kra}), $\Push(x)$ is pseudo-Anosov on $S_{\gamma}$. 
We have to prove that $\Push(x)$, as an element of $\MCG(S,\star)$, has the desired decomposition.

\begin{remark*}
Even though $\Push(x)$ is pseudo-Anosov on $S_{\gamma}$, it does not follow automatically that 
it is pseudo-Anosov on the subsurface $S_{\gamma} \subset S$ regarded as an element of $\MCG(S,\star)$. 
For example one can easily imagine a homeomorphism $\phi$ on $S$ which is trivial in $\MCG(S,\star)$ but
is pseudo-Anosov on some smaller subsurface containing the support of $\phi$. 
In addition, it also may happen that some boundary loops of $S_{\gamma}$ are homotopic, 
and hence the set of boundary loops of $S_{\gamma}$ cannot be taken as the canonical reduction system.
\end{remark*}

Let $C$ be a set of connected components of $\partial S_{\gamma}$. 
By Lemma~\ref{lem:annulus} and Lemma~\ref{lem:essential}, we can write 
$$C = \{\partial_1 S_{\gamma}, \partial_1^{'} S_{\gamma},\ldots,\partial_t S_{\gamma}, \partial_t^{'} S_{\gamma},
\partial_{t+1} S_{\gamma}, \partial_{t+2} S_{\gamma},\ldots,\partial_u S_{\gamma}\},$$ 
such that for each $0<i<t+1$ loops $\partial_i S_{\gamma}$ and $\partial_i^{'} S_{\gamma}$ are homotopic, 
and the set $C^0 = \{\partial_1 S_{\gamma},\ldots,\partial_u S_{\gamma}\}$ consists
of pairwise non-homotopic loops. 

We prove that $C^0$ is the canonical reduction system for $\Push(x)$.
To do that, it is enough to check the first and the second conditions of Proposition~\ref{prop:crs.char}.
Let $S'_{\gamma}$ be the connected component of $S \backslash C^0$ which contains $\gamma$. 
The surface $S'_{\gamma}$ is just $S_{\gamma}$ with annuli attached to some boundary components. 
Thus if $\Push(x)$ is pseudo-Anosov on $S_{\gamma}$, it is pseudo-Anosov on $S'_{\gamma}$. 
On any other component of $S \backslash C^0$, $\Push(x)$ is trivial.

Let us check the second condition. Let $C' \subset C^0$, $\partial_i S_{\gamma} \not\in C'$ and 
$S'$ a connected component of $S \backslash C'$ which contains $\gamma$. 
Since $\partial_i S_{\gamma} \subset S'$, $\Push(x)$ fixes $\partial_t S_{\gamma}$.
All components of the boundary of $S'$ are in $C^0$, thus $\partial_t S_{\gamma}$ is not 
homotopic to any component of the boundary of $S'$. Thus $\partial_t S_{\gamma}$ is essential.
Since $\Push(x)$ is nontrivial and fixes an essential curve, it is reducible.
It follows, that $C'$ does not satisfy the second condition. Hence $C^0$ is the canonical reduction system.
Thus the decomposition of $\Push(x)$ consists of a single pseudo-Anosov diffeomorphism on $S'_{\gamma}$.
\end{proof}

\subsection{Distortion of elements in the $\Aut$-norm of surface groups.}\label{sub:surface.main}
\begin{theorem}\label{thm:surface.main}
Let $x\in\Bo\Gamma_g$. Then either 
\begin{enumerate}[leftmargin=0.5cm,itemindent=.5cm,labelwidth=\itemindent,labelsep=0cm,align=left,label=\alph*)]
\item $x$ is a power of a simple non-separating loop, then the cyclic subgroup generated by $x$ is bounded in the $\Aut$-norm, or
\item $x$ is undistorted in the $\Aut$-norm. Moreover, there exists a homogeneous quasimorphism bounded 
on the set of all simple elements and is non-trivial on $x$.
\end{enumerate}
\end{theorem}

We start with transferring the problem from finding a suitable quasimorphism 
on $\Gamma_g$ to finding a quasimorphism on (a finite index subgroup) of $\MCG(S)$. 
Let us consider the general case first. Let $G$ be a group and let $\Aut(G)$ be the group of automorphisms of $G$. 
Let $c_g(x) = gxg^{-1}$ be the inner automorphism induced by $g$. 
Define homomorphism $c\colon G \to\Aut(G)$ by $c(g) = c_g$. 

\begin{lemma}\label{lem:aut.imp.undist}
Let $G$ be an $H$-finitely generated group for some $H<\Aut(G)$, and let $H_{\circ}<H$ be a finite index subgroup such that $\im(c)<H_{\circ}$. 
Let $x\in G$ and assume that there exists a homogeneous quasimorphism $q\colon H_{\circ}\to\Bo R$ such that $q(c_x) \neq 0$. 
Then $x$ is undistorted in $|\mathord{\cdot}|_H$.
\end{lemma}

\begin{proof} 
Let $\bar{S}$ be an $H$-finite generating set of $G$. 
Then $\bar{S}$ is $H_{\circ}$-finite since $\#(\bar{S}/H_{\circ})=\#(H/H_{\circ})\#(\bar{S}/H)$. 
Thus $\bar{S}$ can be used to define both norms $|\mathord{\cdot}|_H$ and $|\mathord{\cdot}|_{H_{\circ}}$.
To prove that $x$ is undistorted in $|\mathord{\cdot}|_H$ it is enough to prove that it is undistorted in $|\mathord{\cdot}|_{H_{\circ}}$. 

Let $\widetilde{q}$ be the pull back of $q$ to $G$, i.e., $\widetilde{q}(w) = q(c_w)$ for $w \in G$.
Now we will show, that $\widetilde{q}$ is $H_{\circ}$-invariant. Let $w\in G$ and $\psi\in H_{\circ}$.
Note that $\psi c_w\psi^{-1}=c_{\psi(w)}$ and that $q$ is constant on conjugacy classes of $H_{\circ}$. We have
$$
\widetilde{q}(\psi(w))=q(c_{\psi(w)})=q(\psi c_w \psi^{-1})=q(c_w)= \widetilde{q}(w).
$$
We apply Corollary~\ref{cor:H-inv} and finish the proof. 
\end{proof}

We will use this lemma in the case of $G=\Bo\Gamma_g$ and $H=\Aut(\Bo\Gamma_g)$.
We need to take a finite index subgroup $H_{\circ}$, because in most cases there is no 
homogeneous quasimorphism on the whole group $\Aut(\Bo\Gamma_g)$ which is non-zero on a given element $c_x$. 
We will be able to find such quasimorphism on some $H_{\circ}$ and 
conclude undistortedness of $x$ in $|\mathord{\cdot}|_{\Aut(\Bo\Gamma_g)}$.

\begin{proof}[Proof of Theorem \ref{thm:surface.main}]
\textit{Case 1}.
Let $\alpha$ and $\gamma$ be the elements shown in the Figure \ref{fig:alpha}.
We first prove that the cyclic group generated by $\alpha$ is bounded in $|\mathord{\cdot}|_{\Aut(\Bo\Gamma_g)}$. 
Let $|.|_s$ be the simple loops norm.
It is a simple observation, that the loop $\alpha^n\gamma$ is simple for each $n\in\Bo Z$.
Thus we have 
$$|\alpha^n|_s = |\alpha^n\gamma\gamma^{-1}|_s \leq 2.$$
Since $|\mathord{\cdot}|_s$ defines bi-Lipschitz equivalence class $|\mathord{\cdot}|_{\Aut(\Bo\Gamma_g)}$, 
the cyclic subgroup $\langle \alpha \rangle$ is bounded in the $\Aut$-norm.
\begin{figure}[htb]
\centerline{\includegraphics[width=0.4\textwidth]{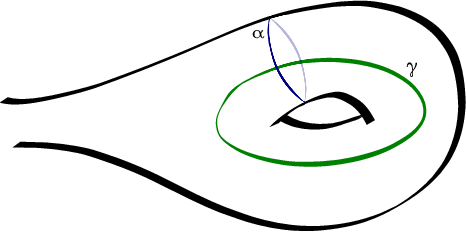}}
\caption{\label{fig:alpha} Loops $\alpha$ and $\gamma$.}
\end{figure}

Now assume that $x\in\Bo\Gamma_g$ is a power of a simple non-separating loop.
Note that every simple non-separating loop can be mapped to $\alpha$.
Indeed, for every two simple non-separating loops $\beta_1$, $\beta_2$, the surfaces $S\backslash\beta_i$ are homeomorphic,
thus all simple non-separating loops are in one $\Aut(\Bo\Gamma_g)$-orbit.
It follows that $x$ can be mapped to $\alpha^m$ for some $m\in\Bo N$.
Thus the subgroup generated by $x$ is bounded.

\textit{Case 2.} Assume that $x\in\Bo\Gamma_g$ is not a power of a simple non-separating loop. 
Let $S$ be a closed surface without punctures and let $\star\in\Bo\Gamma_g$.
Consider the natural map 
$$\Psi \colon \MCG(S,\star) \to \Aut(\Bo\Gamma_g)$$ 
induced by the action of a homeomorphism on the fundamental group of $(S,\star)$.
Let $\Aut^+(\Bo\Gamma_g)\coloneqq \im(\Psi)$, then $\Aut^+(\Bo\Gamma_g)$ has index $2$ in $\Aut(\Bo\Gamma_g)$. 
By the Baer-Dehn-Nielsen theorem, $\Psi$ is an embedding.

It is easy to see (at least for simple elements, see Example \ref{e:push.simple}) that $\Push(w)$ induces a conjugation on $\Bo\Gamma_g$ by $w$. 
Hence instead of working with $\Aut^+(\Bo\Gamma_g)$ and homomorphism $c$, we work with $\MCG(S,\star)$ and $\Push$. 

By the work of Bestvina-Bromberg-Fujiwara (\cite{BBF}) we know that there are plenty of quasimorphisms on mapping class groups.
We describe their result in the way that is convenient for us. 
The group $\MCG(S,\star)$ acts on $H_1(S,\Bo Z/ 3\Bo Z)$. 
Let $H_{\circ}$ be the subgroup of $\MCG(S,\star)$ which contains all elements that act trivially. 
Hence $H_{\circ}$ is a finite index subgroup of $\MCG(S,\star)$. Since conjugation acts trivially on homology, $\Push(\Bo\Gamma_g)<H_{\circ}$.
It follows from\cite[Corollary 5.3 and Corollary 5.5]{BBF}
that there exists a homogeneous quasimorphism on $H_{\circ}$ which is non-zero on $\Push(x)$ if one of the following holds:
\begin{enumerate}[itemsep=-0.2cm,leftmargin=0.5cm,itemindent=.5cm,labelwidth=\itemindent,labelsep=0cm,align=left,label=\alph*)]
\item In the Nielsen-Thurston decomposition of $\Push(x)$ there is at least one pseudo-Anosov element.
\item In the Nielsen-Thurston decomposition of $\Push(x)$ there is at least one non-trivial power of a Dehn twist along some separating curve. 
\end{enumerate}

Now we finish the proof. If $x\in\Bo\Gamma_g$ is not a power of non-separating simple loop, then by Theorem~\ref{thm:gen.kra} either 
\begin{enumerate}[itemsep=-0.2cm,leftmargin=0.5cm,itemindent=.5cm,labelwidth=\itemindent,labelsep=0cm,align=left,label=\alph*)]
\item $x$ is not simple, and the Nielsen-Thurston decomposition contains exactly one pseudo-Anosov element, or
\item $x$ is a power of a Dehn twist along separating loop, and then the Nielsen-Thurston decomposition is 
$$\Push(x) = T_{\gamma^+}^nT_{\gamma^-}^n$$ 
for some separating simple loop $\gamma$ and $n\in\Bo N$.  
\end{enumerate}
In both cases we have a homogeneous quasimorphism 
$$q\colon H_{\circ}\to\Bo R$$ 
which is non-trivial on $\Push(x)$. 
Since $H_{\circ}$ is a finite index subgroup of $\MCG(S,\star)$ it can be viewed as finite index subgroup of $\Aut(\Bo\Gamma_g)$. 
Using Lemma~\ref{lem:aut.imp.undist} we conclude the proof of the theorem.
\end{proof}

Let $S_n$ be the set of all elements of $\Bo\Gamma_g$ which are represented by 
curves with crossing-number at most $n$ (see \cite[Definition~1.1]{C-08})
and let $S_n'$ be the set of all primitive elements in $S_n$. 
In \cite{C-08} D. Calegari proved that $x \in \Bo\Gamma_g$ is undistorted in $|\mathord{\cdot}|_{S_n}$ 
if it has a non-zero self-intersection number. 
He asked (see Question 1.6) whether simple elements are undistorted 
with respect to the metrics $|\mathord{\cdot}|_{S'_n}$.

First of all, we note that all these metrics are bi-Lipschitz equivalent. 
Indeed, by \cite[Remark 1.4]{C-08} all sets $S'_n$ are $\Aut(\Bo\Gamma_g)$-finite.  
The proof is analogous to the proof of the fact that $S_0$, 
which is the set of all simple elements, is $\Aut(\Bo\Gamma_g)$-finite. 
It means that all the metrics~$|\mathord{\cdot}|_{S'_n}$
define the same bi-Lipschitz equivalence class $|\mathord{\cdot}|_{\Aut(\Bo\Gamma_g)}$. 
Thus the part of Theorem~\ref{thm:surface.main} that concerns simple elements gives a complete answer 
to the question of D. Calegari, i.e., we proved the following 

\begin{corollary}\label{cor:calegari-question} 
Simple separating elements in $\Bo\Gamma_g$ are undistorted 
with respect to the $|\mathord{\cdot}|_{S'_n}$-norm for every $n$, 
and simple non-separating elements generate bounded cyclic subgroup.
\end{corollary}

\subsection{More applications and remarks.} 
\label{ss:applications}
\begin{theorem}[Many fellows property]\label{fac:many.fellows.prop.}
Let $S$ be a closed hyperbolic surface of genus $g$ and $l$ be a closed simple non-separating geodesic. 
Then for every $p\in l$ and every neighborhood $U$ of $p$, there is another simple closed geodesic $l'$ 
passing through $U$. 
\end{theorem}

\begin{proof}
Assume on the contrary that there exists $p\in l$ and some neighborhood $U$ of $p$ such that
every closed geodesic different from $l$ does not pass through $U$.
Let $\Bo\Gamma_g$ be the fundamental group of $S$. 
For every $x\in\Bo\Gamma_g$ there is a unique closed geodesic $l_x$ in the free homotopy class of $x$.
Two geodesics $l_x$ and $l_{x'}$ are equal if and only if $x$ and $x'$ are conjugated.
Moreover, simple elements of $\Bo\Gamma_g$ correspond to simple closed geodesics.  
Let $l = l_{x_0}$. Let $\alpha$ be a $1$-differential form supported on $U$ such that $\int_{l}\alpha\neq 0$. 
In \cite{barge.ghys} Barge-Ghys showed that the following function:
$$
q(x) = \int_{l_x} \alpha,
$$
is a homogeneous quasimorphism. 
If $x$ is a simple element not conjugated to $x_o$, 
then $l_{x}\neq l_{x_o}$ and by our assumption $l_x$ does not pass through $U$.
Hence $q(l_x) = 0$. It follows that the only simple elements on which $q$ is non-zero are conjugates of $x_0$. 
Since $q$ is constant on conjugacy classes, it is bounded on the set of all simple elements. 
By Lemma~\ref{lem:qm.undist} we get that $x_0$ is undistorted in the simple loops norm 
which contradicts Theorem~\ref{thm:surface.main}   
\end{proof}

\begin{remark}
Theorem~\ref{fac:many.fellows.prop.} holds for separating closed geodesics as well, see \cite[Lemma~5.1]{parlier}. 
Note that unlike the set of all closed geodesics, the set of simple closed geodesics is not dense in $S$ and,
as suggested by Theorem~\ref{fac:many.fellows.prop.}, there are some preferred tracks chosen by simple geodesics. 
For a further discussion of this phenomenon see \cite{buser.parlier}.
\end{remark}

\begin{remark}
Let $x$ be a simple non-separating element and let $\gamma$ be a loop that represents $x$. 
Then there is no finite index subgroup of $\MCG(S,\star)$ containing $\Push(x)$
such that $T_{\gamma^+}$ is not conjugated to $T_{\gamma^-}$. Indeed, otherwise one 
could find a quasimorphism on $\MCG(S,\star)$ which is non-trivial on $\Push(x)$.
Using Lemma~\ref{lem:aut.imp.undist} we see that $x$ would be undistorted in the simple loops norm which
contradicts Theorem~\ref{thm:surface.main}.
\end{remark}

\section{$\Aut(\Bo F_2)$-invariant quasimorphisms on $\Bo F_2$.} 

Let $S_g$ be a closed surface of genus $g$. 
Let $\circ,\star \in S_g$ be two arbitrary points.
We shall regard $\circ$ as a puncture and $\star$ as a base point.
Let us consider the group $\MCG(S_g\backslash\{\circ\},\star)$ 
of mapping classes of $S_g\backslash\{\circ\}$ fixing the point $\star$. 

The natural action of a homeomorphism on $\Bo F_{2g}=\pi_1(S_g\backslash \{\circ\},\star)$ induces a map
$$\pi\colon\MCG(S_g\backslash\{\circ\},\star)\to\Aut(\Bo F_{2g}).$$ 
We claim that this map is injective. 
Indeed, the surface $S_g\backslash\{\circ\}$ can be described 
as a regular $4g$-gon with opposite edges identified, such that
the point $\circ$ lies in the center and $\star$ is one of the vertices. 
If $\pi(\psi)$ is the identity for some $\psi~\in~\MCG(S_g\backslash\{\circ\},\star)$, 
then there exists a representation of $\psi$ fixing pointwise the edges of the $4g$-gon.
Thus we can regard this representation of $\psi$ as a homeomorphism of the punctured disc fixing the boundary. 
By the Alexander trick (\cite[Lemma~2.1]{fm}) such an element is isotopic to the identity. 

In the next theorem we consider non-simple element $x\in\Bo F_{2g}$ 
such that for every $\psi\in\MCG(S_g\backslash\{\circ\},\star)$ we have $x^{-1}\neq\psi(x)$.
We postpone the proof of the existence of such elements to the next section (see Lemma~\ref{lem:not.inverted}).  

\begin{theorem}\label{thm:F.inv}
Let $x\in\Bo F_{2g}$ such that for every $\psi\in\MCG(S_g\backslash \{\circ\},\star)$ we have $x^{-1}\neq\psi(x)$, 
and $x$ cannot be represented by a simple loop in $S_g\backslash\{\circ\}$.
Then there exists a non-trivial $\MCG(S_g\backslash \{\circ\},\star)$-invariant 
homogeneous quasimorphism on $\Bo F_{2g}$ which is non-zero on $x$. 
\end{theorem}

\begin{proof}
We consider the Birman embedding $\Bo F_{2g}\xrightarrow{\Push}\MCG(S_g\backslash\{\circ\},\star)$.
Since $x$ is not simple, it follows from Theorem~\ref{thm:gen.kra} that the Nielsen-Thurston 
decomposition of $\Push(x)$ consists of one non-trivial pseudo-Anosov pure component.
In addition, $\Push(x)$ is not conjugated to its inverse in the group $\MCG(S_g\backslash \{\circ\}, \star)$.
Indeed, if there is an element $\psi\in\MCG(S_g\backslash\{\circ\},\star)$ which conjugates $\Push(x)$ to its inverse, 
then it implies
$$\Push(\psi(x)) = \psi \Push(x) \psi^{-1} = \Push(x)^{-1} = \Push(x^{-1}).$$ 
Since $\Push$ is injective, we conclude that $\psi(x) = x^{-1}$ which contradicts our assumption. 

Now we use quasimorphisms constructed in \cite{BBF}. 
Note that if an element in a group is conjugated to its inverse, 
then every homogeneous quasimorphism vanishes on this element. 
If follows from \cite[Theorem~4.2]{BBF} that for pure elements of $\MCG(S_g\backslash\{\circ\},\star)$, 
being conjugated to its inverse is the only obstruction to be detected by homogeneous quasimorphisms.
Due to Lemma~\ref{lem:aut.imp.undist}, $\Push(x)$ is pure. 
Thus there exists a homogeneous quasimorphism $q$ on $\MCG(S_g\backslash\{\circ\},\star)$ which is non-zero on $\Push(x)$.
The pull-back of $q$ to $\Bo F_{2g}$ by $\Push$ gives us an $\MCG(S_g\backslash\{\circ\},\star)$-invariant 
quasimorphism which does not vanish on $x$.
\end{proof}

\begin{remark}
Let $n\geq 3$. One would like to construct $\Aut(\Bo F_n)$-invariant quasimorphism on $\Bo F_n$
by restricting a quasimorphism from $\Aut(\Bo F_n)$ to $\Bo F_n$ which is embedded in $\Aut(\Bo F_n)$ via inner automorphisms.
However, despite an extensive study of $\Aut(\Bo F_n)$  
it is not known if there are quasimorphisms on $\Aut(\Bo F_n)$ which restrict non-trivially to $\Bo F_n$.
\end{remark}

\begin{theorem}\label{thm:mcg.inf.dim}
The linear space of $\MCG(S_g\backslash\{\circ\},\star)$-invariant homogeneous quasimorphisms on $\Bo F_{2g}$ is infinite dimensional. 
\end{theorem}
\begin{proof}
It follows from Lemma~\ref{lem:different.orbits} and Lemma~\ref{lem:not.inverted} that there is an infinite sequence $z_1,z_2,\ldots$ 
of integers such that the elements $x_k = a^{z_k}b^{2z_k}a^{3z_k}b^{4z_k}$ have the following properties:

\begin{enumerate}[leftmargin=0.5cm,itemindent=.5cm,labelwidth=\itemindent,labelsep=0cm,align=left,label=\alph*)]
\item $x_k$ and $x_k^{-1}$ belong to different $\Aut(\Bo F_{2g})$-orbits, 
\item for $i,j \in \Bo Z \backslash\{0\}$ and $k \neq k'$, elements $x_k^i$ and $x_{k'}^j$ belong to different $\Aut(\Bo F_{2g})$-orbits.
\end{enumerate}

Elements $x_k$ are not simple. 
Indeed, every element $x\in\Bo F_{2g}$ that can be represented by a simple loop in $\MCG(S_g\backslash\{\circ\},\star)$
is a primitive element of $\Bo F_{2g}$, and hence is inverted by some automorphism of $\Bo F_{2g}$. 

Recall that by Theorem~\ref{thm:gen.kra}, $\Push(x_k)$ is a pure mapping class, 
whose Nielsen-Thurston decomposition consists only of one pseudo-Anosov component. 
Moreover, using an argument from the proof of Theorem~\ref{thm:F.inv}, 
property a) implies that the element $\Push(x_k)$ is not conjugated to its inverse in the mapping class group, 
and property b) implies that for $k \neq k'$ and any non-zero $i$ and $j$, $\Push(x_k)^i$ is not conjugated to $\Push(x_{k'})^j$. 
In the language of chiral and achiral classes introduced in \cite{BBF}, it means that the elements $\Push(x_k)$ 
represent different chiral classes for different $k$. 
Let $k\in\Bo N$. It follows from \cite[Proposition~4.4]{BBF} that each function 
$$\{\Push(x_1),\ldots,\Push(x_k)\} \to\Bo R$$
is a restriction of a homogeneous quasimorphism on $\MCG(S_g\backslash\{\circ\},\star)$.
If we pull-back these quasimorphisms to $\Bo F_{2g}$ by $\Push$, we obtain that 
each function 
$$\{x_1,\ldots,x_k\} \to \Bo R$$ 
is a restriction of some $\MCG(S_g\backslash\{\circ\},\star)$-invariant homogeneous quasimorphism.
Consequently, for each $k\in\Bo N$ we constructed a $k$-dimensional subspace 
of $\MCG(S_g\backslash\{\circ\},\star)$-invariant homogeneous quasimorphisms.
\end{proof}

For $g>1$, $\pi(\MCG(S_g\backslash\{\circ\},\star))$ has infinite index in $\Aut(\Bo F_{2g})$.
The situation is different for $g=1$. For completeness we give a proof of the following lemma.
 
\begin{lemma}\label{lem:F2.mcg.in.aut}
$\pi(\MCG(S_1\backslash \{\circ\},\star))$ has index $2$ in $\Aut(\Bo F_2)$.
\end{lemma}
\begin{proof}
Let us extend the group $\MCG(S_1\backslash \{\circ\},\star)$ to $\MCG^{\pm}(S_1\backslash \{\circ\}, \star)$
by allowing orientation reversing homeomorphisms.
Then $\pi$ naturally extends to the map $\pi'\colon\MCG^{\pm}(S_1\backslash \{\circ\},\star) \to\Aut(\Bo F_2)$.
By the same argument as in the beginning of this chapter, $\pi'$ is injective. 
Let $\{a,b\}$ be a basis of $\Bo F_2$.
The group $\Aut(\Bo F_2)$ is generated by the following automorphisms:
$$
\begin{tikzcd}[row sep = 0ex,
/tikz/column 1/.append style={anchor=base west},
/tikz/column 2/.append style={anchor=base west},
/tikz/column 3/.append style={anchor=base west},
/tikz/column 4/.append style={anchor=base west}]
a \to a^{-1} &\hspace{2pt}a \to b &\hspace{2pt}a \to ab\\ 
b \to b      &\hspace{2pt}b \to a &\hspace{2pt}b \to b. 
\end{tikzcd}
$$
Each one of them can be realized by an element of $\MCG^{\pm}(S_1\backslash \{\circ\},\star)$.
Thus $\pi'$ is onto. 
\end{proof}

Every automorphism of $\Bo F_n$ acts on its abelianisation which is isomorphic to $\Bo Z^n$, 
thus defines a matrix over $\Bo Z$. 
Let $\Aut^+(\Bo F_n)$ be the subgroup consisting of all elements which define matrices of determinant $1$. 
In the case of $n=2$, we have $\Aut^+(\Bo F_2)=\pi(\MCG(S_1\backslash\{\circ\},\star))$. 
An immediate consequence of Theorem~\ref{thm:mcg.inf.dim} is that the linear space of homogeneous 
$\Aut^+(\Bo F_2)$-invariant quasimorphisms is infinite dimensional. 
\begin{remark}
The fact that the space of homogeneous $\Aut^+(\Bo F_2)$-invariant quasimorphisms on $\Bo F_2$ is non-trivial 
was recently proved in his thesis by Huber in \cite{Huber}. 
He showed that certain rotation number quasimorphism is $\Aut^+(\Bo F_2)$-invariant.
\end{remark}

In the next corollary we improve this result to $\Aut(\Bo F_2)$-invariant homogeneous quasimorphisms.

\begin{corollary}\label{cor:F_2.inv}
The linear space of homogeneous $\Aut(\Bo F_2)$-invariant quasimorphisms on $\Bo F_2$ is infinite dimensional.
\end{corollary}

\begin{proof}
Let $a$ and $b$ be generators of $\Bo F_2$. 
Denote by $\sigma$ the automorphism defined by $\sigma(a) = a^{-1}, \sigma(b) = b$. 
Let $\{z_i\}_{i=1}^\infty$ be a sequence of integers such that $\OP{gcd}(z_i)\neq \OP{gcd}(z_j)$ for $i \neq j$. 
Let $x_k = a^{z_k}b^{2z_k}a^{3z_k}b^{4z_k}$. 
Consider the set 
$$X:=\{x_1, \sigma(x_1), x_2, \sigma(x_2), \ldots \}\subset\Bo F_2.$$
It follows from Lemma~\ref{lem:not.inverted} that no $x_k$ is inverted by an automorphism of $\Bo F_2$. 
The same applies to elements $\sigma(x_k)$. It means that each $\Push(\sigma(x_k))$ is chiral. 

For every $x,y\in X$, the elements $\Push(x)$ and $\Push(y)$ define 
different chiral classes in $\MCG(S_1\backslash\{\circ\},\star)$, see Lemma~\ref{lem:different.orbits2}.
It follows from \cite[Proposition~4.4]{BBF} that each function 
$$\{\Push(x_1),\Push(\sigma(x_1)),\ldots,\Push(x_k),\Push(\sigma(x_k))\}\to\Bo R$$ 
is a restriction of a homogeneous quasimorphism on $\MCG(S_1\backslash\{\circ\},\star)$. 
By pulling back to $\Bo F_2$, we obtain that every function 
$$\{x_1,\sigma(x_1),\ldots,x_k,\sigma(x_k)\}\to\Bo R$$ 
is a restriction of a homogeneous $\Aut^+(\Bo F_2)$-invariant quasimorphism.

Let $Q(\Bo F_2)^{\Aut^+}$ be the space of homogeneous $\Aut^+(\Bo F_2)$-invariant quasimorphisms on $\Bo F_2$. 
Let $X_k = \{x_1,\sigma(x_1),\ldots,x_k,\sigma(x_k)\}$. 
Denote by $\Bo R^{X_k}$ the set of all functions from $X_k$ to $\Bo R$.
We have the following commutative diagram:
$$
\begin{tikzcd}
Q(\Bo F_2)^{\Aut^+} \arrow{r}{\Sym_{\sigma}} \arrow[two heads]{d} & Q(\Bo F_2)^{\Aut^+} \arrow[two heads]{d}\\
\Bo R^{X_k} \arrow{r}{\widetilde{\Sym}_{\sigma}} & \Bo R^{X_k}, 
\end{tikzcd}
$$
where 
$$\Sym_{\sigma}(q)(x) = q(x) + q(\sigma(x)),$$ 
for $q\in\bar{Q}(\Bo F_2)^{\Aut^+}$. 
We claim that $\Sym_{\sigma}(q)$ is an $\Aut(\Bo F_2)$-invariant homogeneous quasimorphism. 
It is clear that $\Sym_{\sigma}(q)$ is a quasimorphism, because $q$ is a quasimorphism and $\sigma$ is an 
automorphism of $\Bo F_2$. 
To prove the $\Aut(\Bo F_2)$ invariance, we first note that 
$\Sym_{\sigma}(q)$ is $\sigma$-invariant, which is obvious from the definition.
Moreover, for every $\psi \in\Aut^+(\Bo F_2)$ we can 
find $\psi' \in\Aut^+(\Bo F_2)$ such that $\psi'\sigma = \sigma\psi$. Now
\begin{align*}
\Sym_{\sigma}(q)(\psi(x))&=q(\psi(x))+q(\sigma\psi(x))=q(x)+q(\psi'\sigma(x))\\
&=q(x)+q(\sigma(x))=\Sym_{\sigma}(q)(x).
\end{align*}
Thus the quasimorphism $\Sym_{\sigma}(q)$ is $\sigma$-invariant 
and $\Aut^+(\Bo F_2)$-invariant, and consequently $\Aut(\Bo F_2)$-invariant.

The map $\widetilde{\Sym}_{\sigma}$ is defined by $\widetilde{\Sym}_{\sigma}(f)(x) = f(x) + f(\sigma(x))$.
The vertical epimorphisms in the above diagram are restrictions. 
We have that 
$$\OP{Im}(\widetilde{\Sym}_{\sigma}) = \{ f \in \Bo R^{X_k} \colon f(x_l)=f(\sigma(x_l))\,\text{for each }\thinspace l\in\{1,\ldots,k\}\}$$
is a $k$-dimensional linear space. Each element of $\OP{Im}(\widetilde{\Sym}_{\sigma})$ is a restriction of $\Sym_{\sigma}(q)$ for some
$q\in Q(\Bo F_2)^{\Aut^+}$, which is an $\Aut(\Bo F_2)$-invariant homogeneous quasimorphism.
Thus for each $k\in\Bo N$ the space of $\Aut(\Bo F_2)$-invariant homogeneous quasimorphisms on $\Bo F_2$ contains
a $k$-dimensional subspace. 
\end{proof}

\begin{remark}\label{R:inf-not-Brooks}
In his thesis Hase \cite{Hase-2, Hase-1} proved that the space of 
quasimorphisms on $\Bo F_2$ that are not $\Aut(\Bo F_2)$-invariant  is dense 
in the space of all homogeneous quasimorphisms on $\Bo F_2$. 
In particular, finite linear combinations of counting quasimorphisms are not $\Aut(\Bo F_2)$-invariant. 
Hence a rotation number quasimorphism considered by Huber in \cite{Huber} 
is not $\Aut(\Bo F_2)$-invariant, since it is a linear combination of counting quasimorphisms.

It follows from Corollary~\ref{cor:F_2.inv}, 
that there exists an infinite dimensional space of quasimorphisms on $\Bo F_2$ where each quasimorphism 
can not be expressed as a finite linear combination of counting quasimorphisms. 
On the other hand, Grigorchuk \cite{grigo} showed that every quasimorphism 
is a linear combination of (possibly infinitely many) counting quasimorphisms. 
\end{remark}

\section{Whitehead algorithm and some elements of $\Bo F_n$.}
J.~H.~C.~Whitehead \cite{wh2} described an algorithm that, given two elements $x,x'\in\Bo F_n$, 
it decides if there is $\psi\in\Aut(\Bo F_n)$ for which $\psi(x)=x'$. 
Below we recall the Whitehead algorithm. 

\begin{definition}
Let $X$ be a basis of $\Bo F_n$. An element $\psi\in\Aut(\Bo F_n)$ is called
\begin{enumerate}[leftmargin=0.5cm,itemindent=.5cm,labelwidth=\itemindent,labelsep=0cm,align=left,label=\alph*)]
\item  \textbf{Permutation automorphism} if $\psi$ permutes the set $X \cup X^{-1}$.
\item \textbf{Whitehead automorphism} if there is an element $a \in X \cup X^{-1}$ such 
that $\psi(a)=a$ and $\psi(x)\in\{x,ax,xa^{-1},axa^{-1} \}$ for each $x\in X\backslash\{a\}$.
\end{enumerate}
\end{definition}

Let $\Omega_n$ denote the set of all permutation automorphisms. 
The set $\Omega_n$ is a finite subgroup of $\Aut(\Bo F_n)$ 
which is isomorphic to the extended permutation group. 

\begin{theorem}[Whitehead]\label{thm:whitehead.alg}
Let $x\in\Bo F_n$ and let $m=\min\{|\psi(x)|_X\}$ where the minimum is taken over all $\psi\in\Aut(\Bo F_n)$. 
If $|x|_X>m$, then there exists a Whitehead automorphism $h$ such that $|h(x)|_X<|x|_X$. 
If $x$ and $x'$ are in the same $\Aut$-orbit and $|x|_X=|x'|_X=m$, 
then there exists a sequence of permutation and Whitehead automorphisms $t_1,\ldots,t_l$ such that:
\begin{enumerate}[leftmargin=0.5cm,itemindent=.5cm,labelwidth=\itemindent,labelsep=0cm,align=left,label={}]
\item $t_l\ldots t_1(x) = x'$ and
\item $|x|_X = |t_1(x)|_X = |t_2t_1(x)|_X = \ldots = |t_l\ldots t_1(x)|_X = m.$
\end{enumerate}
\end{theorem}
 
Let $x\in\Bo F_n$. We denote by $\xoverline{x}$ the conjugacy class represented by $x$. 
If $c$ is any conjugacy class, we define its length by 
$$|c|_X:=\min\{|x|_X,\thinspace\xoverline{x}=c\}.$$ 
Note that $\Aut(\Bo F_n)$ acts on conjugacy classes of $\Bo F_n$. 
It is easy to see, that the analogous version of Whitehead algorithm 
works for conjugacy classes and the norm defined above. 

In the following lemmas we consider a sequence $\{x_k\}_{k=2}^\infty$ where $x_k\in \Bo F_n$ is of the form  
$$x_k = a^kb^{2k}a^{3k}b^{4k}.$$
Here the elements $a$ and $b$ denote two different generators of~$\Bo F_n$. 

\begin{lemma}\label{lem:permutation}
Let $n=2$ and $X = \{a,b\}$.
Let $t$ be a permutation automorphism. 
Then $t(\xoverline{x_k}) = \xoverline{x_k}$ if and only if $t$ is the identity, i.e., 
the group $\Omega_2$ acts freely on the orbit $\Omega_2(\xoverline{x_k})$.  
\end{lemma}

\begin{proof}
The conjugacy class $t(\xoverline{x_k})$ is represented by the element 
of the form $u^kv^{2k}u^{3k}v^{4k}$ for some $u,v\in\{a,b,a^{-1},b^{-1}\}$. 
This element represents conjugacy class of $a^kb^{2k}a^{3k}b^{4k}$ if and only if $u=a$ and $v=b$. 
It means that $t(a)=a$ and $t(b)=b$. Thus $t$ is the identity. 
\end{proof}

\begin{lemma}\label{lem:whitehead.aut}
Let $\psi$ be a Whitehead automorphism and let $\xoverline{x}\in\Omega_n(x_k)$. 
Then either $\psi(\xoverline{x}) = \xoverline{x}$ or $|\psi(\xoverline{x})|_X>|\xoverline{x}|$. 
\end{lemma}

\begin{proof}
The element $\xoverline{x}$ equals to $\phi(x_k)$ for some $\phi\in\Omega_n$. 
Thus $\xoverline{x}$ is represented by the element $u^kv^{2k}u^{3k}v^{4k}$, 
where $u,v \in X \cup X^{-1}$ and $u\notin \{v,v^{-1}\}$. 
Consider a Whitehead automorphism $\psi$. 
Assume that $a$ from the definition of the Whitehead automorphism is not equal to $u,v,u^{-1},v^{-1}$.
If $\psi(u) = aua^{-1}$ and  $\psi(v) = ava^{-1}$, then $\psi(\xoverline{x})=\xoverline{x}$.
In all other cases we have $|\psi(\xoverline{x})|_x > |\xoverline{x}|_x$, 
since in $\psi(\xoverline{x})$ there always will be some occurrences of the letter $a$.
 
If $a \in \{u,v,u^{-1},v^{-1}\}$, 
then up to inner automorphisms, there are only 5 different 
ways a Whitehead automorphism can act on $\{u,v\}$. They are listed below:
$$
\begin{tikzcd}[row sep = 0ex,
/tikz/column 1/.append style={anchor=base west},
/tikz/column 2/.append style={anchor=base west},
/tikz/column 3/.append style={anchor=base west},
/tikz/column 4/.append style={anchor=base west}]
u \to u &\hspace{2pt}u \to uv^{\pm} &\hspace{2pt}u \to u\\ 
v \to v &\hspace{2pt}v \to v        &\hspace{2pt}v \to vu^{\pm}. 
\end{tikzcd}
$$
Direct computation shows that these automorphisms, 
except the one which fixes $u$ and $v$, increase the length of $\xoverline{x}$ (provided that $k>1$).  
Thus $\psi$ does not increase the length of $\xoverline{x}$ if and only if $\psi$ fixes $\xoverline{x}$.
\end{proof}

\begin{lemma}\label{lem:not.inverted}
Elements $x_k$ and $x_k^{-1}$ belong to different $\Aut(\Bo F_n)$-orbits. 
\end{lemma}
\begin{proof}
It follows from Theorem~\ref{thm:whitehead.alg} that 
$\Omega_n(\xoverline{x}_k)$ is the set of all conjugacy classes minimizing the norm 
$|\mathord{\cdot}|_X$ in the $\Aut(\Bo F_n)$-orbit of $\xoverline{x}_k$.
Indeed, if some $\xoverline{y}\in F_n$ minimizes the norm, 
then there exist permutation or Whitehead automorphisms $t_1,\ldots,t_l$ such that 
$$|\xoverline{x}_k|_X = |t_1(\xoverline{x}_k)|_X = \ldots = |t_l\ldots t_1(\xoverline{x}_k)|_X$$ 
and $t_l\ldots t_1(\xoverline{x}_k) = \xoverline{y}$.
Since $|\xoverline{x}_k|_X = |t_1(\xoverline{x}_k)|_X$, we conclude that either $t_1$ is a permutation automorphism, 
or $t_1$ is a Whitehead automorphism and by Lemma \ref{lem:whitehead.aut} we have $t_1(\xoverline{x}_k) = \xoverline{x}_k$. 
Then we apply the same argument to the element $t_1(\xoverline{x}_k)$ 
and the equality 
$$|t_2(t_1(\xoverline{x}_k))|_X = |t_1(\xoverline{x}_k)|_X$$ 
to conclude that $t_2$ is a permutation automorphism or $t_2(t_1(\xoverline{x}_k)) = t_1(\xoverline{x}_k)$. 
It follows that each $t_i$ is a permutation automorphism, or fixes the element $t_{i-1}\ldots t_1(\xoverline{x}_k)$. 
Hence $\xoverline{y}\in\Omega_n(\xoverline{x}_k)$. 

If $\xoverline{x}_k^{-1}$ is in the same $\Aut(\Bo F_n)$-orbit as $\xoverline{x}_k$, 
then $\xoverline{x}_k^{-1}$ would minimize the norm.
Thus to prove the lemma, it remains to prove that $\xoverline{x}_k^{-1}$ does not belong to $\Omega_n(\xoverline{x}_k)$. 
Note that if $t(\xoverline{x}_k)=\xoverline{x}_k^{-1}$ for $t\in\Omega_n$, then necessarily $t(a), t(b) \in \{a,a^{-1},b,b^{-1}\}$. 
It is easy to check that for such automorphisms we always have $t(\xoverline{x}_k) \neq \xoverline{x}_k^{-1}$. 
\end{proof}

\begin{lemma}\label{lem:different.orbits}
Let $i,j \in \Bo Z\backslash\{0\}$ and let $k,l\geq 2$ such that $\OP{gcd}(k) \neq \OP{gcd}(l)$. 
Elements $(x_k)^i$ and~$(x_l)^j$ belong to different $\Aut(\Bo F_n)$-orbits.
\end{lemma}

\begin{proof}
For every element $x\in\Bo F_n$ there is a unique (up to a sign) number $n$ 
and a unique (up to taking an inverse) element $p$ which is not a proper 
power of any other element such that $x = p^n = (p^{-1})^{-n}$.
Assume that there exists an automorphism $\psi\in\Aut(\Bo F_n)$ such that 
$\psi((x_k)^i) = (x_l)^j$. Both $x_k$ and $x_l$ are not proper powers, hence $\psi(x_k) = (x_l)^{\pm 1}$.
In what follows we show that this is impossible. 

Every automorphism $\psi\colon\Bo F_n\to\Bo F_n$ induces the abelianisation automorphism 
$$\Ab(\psi)\colon\Bo Z^n\to\Bo Z^n.$$
Elements $x_k$ and $(x_l)^{\pm}$ are mapped to vectors with coordinates equal to $4k,6k$ and $\pm4l,\pm6l$ in the abelianisation. 
It is enough to show, that these vectors belong to different $\Aut(\Bo Z^n)$-orbits. 
Indeed this is the case, since automorphisms of $\Bo Z^n$ preserve the greatest common divisor of coordinates of a vector, and 
$$\OP{gcd}(4k,6k) = 2\OP{gcd}(k)\neq 2\OP{gcd}(l) = \OP{gcd}(\pm4l,\pm6l).$$
\end{proof}

\begin{lemma}\label{lem:fliping}
Let $n=2$ and $X = \{a,b\}$.
Let $\sigma \in\Aut(\Bo F_2)$ be defined by $\sigma(a) = a^{-1}, \sigma(b) = b$.
Then $x_k$ and $\sigma(x_k)$ belong to different $\Aut^{+}(\Bo F_2)$-orbits. 
\end{lemma}

\begin{proof}
Let us recall that the group $\Aut^+(\Bo F_2)$ consists of automorphisms $\psi$ for which $\OP{det}(\Ab(\psi))=1$.
We have that 
$$\sigma(x_k) = a^{-k}b^{2k}a^{-3k}b^{4k}$$ 
and $\sigma\in\Aut(\Bo F_2)\setminus\Aut^+(\Bo F_2)$.
Suppose that $\sigma(x_k)=\psi(x_k)$ for some automorphism $\psi\in\Aut^+(\Bo F_2)$. 
It means that $\psi^{-1}\sigma\in \OP{Stab}(x_k)$ and $\psi^{-1}\sigma\in\Aut(\Bo F_2)\setminus\Aut^+(\Bo F_2)$.
In what follows we show that $\OP{Stab}(x_k)<\Aut^+(\Bo F_2)$ which is a contradiction.

Let us consider the stabilizer $\OP{Stab}(\xoverline{x}_k)$ of the conjugacy class $\xoverline{x}_k$.
Of course $\OP{Stab}(x_k)<\OP{Stab}(\xoverline{x}_k)$. We show that $\OP{Stab}(\xoverline{x}_k)<\Aut^+(\Bo F_2)$. 
We use the construction presented in \cite{mccool} in order to find a generating set of $\OP{Stab}(\xoverline{x}_k)$. 
First we define a graph $\Delta$ as follows: a vertex of $\Delta$ is a conjugacy class of minimal length in the 
$\Aut(\Bo F_2)$-orbit of $\xoverline{x}_k$. It follows from the proof of 
Lemma~\ref{lem:not.inverted} that this set equals to $\Omega_2(\xoverline{x}_k)$.
Two vertices $v_1,v_2\in\Omega_2(\xoverline{x}_k)$ are connected by a directed edge from $v_1$ to $v_2$ 
if there is a permutation or Whitehead automorphism $\psi$ such that $\psi(v_1) = v_2$.
We will consider edge embedded loops in $\Delta$ based at $\xoverline{x}_k$.
The theorem of McCool \cite{mccool} says, that $\OP{Stab}(\xoverline{x}_k)$ is generated by 
elements which are products of labels read from all possible loops like this. 

It follows from Lemma~\ref{lem:permutation} that the subgraph 
spanned by edges labeled with permutation automorphisms is 
a complete graph on the set $\Omega_2(\xoverline{x}_k)$, with no loops.
Lemma~\ref{lem:whitehead.aut} implies that all edges labeled by Whitehead automorphisms are loops.  

Let $\delta$ be a loop in $\Delta$ based at $\xoverline{x}_k$. 
Now we show that the word $w$ which is read from the labels of $\delta$ is trivial in  
$\Aut(\Bo F_2)/\Aut^+(\Bo F_2)\cong\Bo Z/2\Bo Z$. 
Note that all Whitehead automorphisms belong to $\Aut^+(\Bo F_2)$. 
Thus we can ignore labels coming from the edges labeled by Whitehead automorphisms.
Since the edges labeled by Whitehead automorphisms are loops, 
we can assume that $\delta$ goes through the edges labeled only by permutation automorphisms. 
Thus we can assume, that the element read from the labels of $\delta$ 
is just a permutation automorphism which fixes $\xoverline{x}_k$. 
By Lemma~\ref{lem:permutation} this element is trivial.
\end{proof}

\begin{remark}
More detailed analysis shows that $\OP{Stab}(x_k)$ is a cyclic subgroup generated by conjugation by $x_k$. 
One also can prove the analog of Lemma~\ref{lem:fliping} for $\Bo F_n$ by replacing $x_k$ with more complicated elements.
\end{remark}

\begin{lemma}\label{lem:different.orbits2}
Let $i,j \in \Bo Z\backslash\{0\}$ and let $k,l\geq 2$ such that $\OP{gcd}(k \neq \OP{gcd}(l)$. Then:
\begin{enumerate}[leftmargin=0.5cm,itemindent=.5cm,labelwidth=\itemindent,labelsep=0cm,align=left,label=\alph*)]
\item $x_k^i$ and $\sigma(x_k)^j$ belong to different $\Aut^+(\Bo F_n)$-orbits. 
\item $\sigma(x_k)^i$ and $\sigma(x_l)^j$ belong to different $\Aut(\Bo F_n)$-orbits.
\item $x_k^i$ and $\sigma(x_l)^j$ belong to different $\Aut(\Bo F_n)$-orbits.
\end{enumerate} 
\end{lemma}

\begin{proof}
It follows from the proof of Lemma~\ref{lem:different.orbits} that we can assume $i=1$ and $j=\pm1$. 
Lemma~\ref{lem:fliping} implies a) for $j=1$. 
To prove a) for $j=-1$ it is enough to note that $\sigma(x_k)^{-1}$ and $x_k^{-1}$ are in the same $\Aut(\Bo F_n)$-orbit.
If $x_k$ and $\sigma(x_k)^{-1}$ are in the same $\Aut^+(\Bo F_n)$-orbit, then 
$x_k$ and $x_k^{-1}$ are in the same $\Aut(\Bo F_n)$-orbit, which contradicts Lemma~\ref{lem:not.inverted}. 
The proof of b) and c) is analogous to the proof of Lemma~\ref{lem:different.orbits}.
\end{proof}

\bibliography{bibliography}
\bibliographystyle{plain}

\end{document}